\documentclass[a4paper,11pt]{article}
\usepackage{amsmath,amssymb,amsthm,graphicx}
\usepackage{fullpage}
\usepackage{enumitem}
\newcommand{\gf}[1]{\ensuremath{\mathrm{GF}(#1)}}

\newtheorem{theorem}{Theorem}[section]
\newtheorem{lemma}[theorem]{Lemma}
\newtheorem{corollary}[theorem]{Corollary}
\newtheorem{question}{Question}
\theoremstyle{definition}
\newtheorem{example}[theorem]{Example}
\newtheorem{definition}[theorem]{Definition}
\newtheorem{remark}[theorem]{Remark}
\newtheorem{construction}[theorem]{Construction}
\newtheorem{proposition}[theorem]{Proposition}

\title{Weighted external difference families and R-optimal AMD codes}
\author{S. Huczynska and M.B.~Paterson}
\date{}
\begin{document}
\maketitle

\begin{abstract}
In this paper, we provide a mathematical framework for characterizing AMD codes that are R-optimal.  We introduce a new combinatorial object, the \emph{reciprocally-weighted external difference family (RWEDF)}, which corresponds precisely to an R-optimal weak AMD code.  This definition subsumes known examples of existing optimal codes, and also encompasses combinatorial objects not covered by previous definitions in the literature.  By developing structural group-theoretic characterizations, we exhibit infinite families of new RWEDFs, and new construction methods for known objects such as near-complete EDFs.  Examples of RWEDFs in non-abelian groups are also discussed.
\end{abstract}

\section{Introduction}
Algebraic manipulation detection (AMD) codes were introduced in the cryptographic literature by Cramer, Dodis, Fehr, Padr\'{o} and Wichs as a tool with a range of cryptographic applications.  They are a generalisation of existing approaches to constructing secret sharing schemes secure against cheating \cite{CraDodFehPadWic}.  Considerable attention has been devoted to studying and constructing various types of AMD codes \cite{amdleakage,Cramer2013,CraPadrXing,wangkarpov}.  Paterson and Stinson explored combinatorial properties of AMD codes, including connections with various types of external difference families \cite{PatSti}.  Strong external difference families, which give rise to AMD codes in the {\em strong} model, have received much recent attention \cite{BAO2017,HuczPat,jedwabli,martinstinson,WenEtAl,WenYanFen}.

In this paper we consider the so-called {\em weak} model for AMD codes.  Before giving the definition, we first establish some notation and conventions that we will use throughout the paper. Unless otherwise stated, our groups will be abelian and written additively.  For a group $G$, we denote $G \setminus \{0\}$ by $G^*$ (where $0$ is the identity of $G$).  In studying AMD codes it is necessary to consider differences between group elements occurring in disjoint subsets of an abelian group, and we find it convenient to define the following notation:
\begin{definition}
Let $G$ be a finite abelian group and let $\{ A_1, \ldots, A_m \}$ be a collection of disjoint subsets of $G$.  For $\delta \in G^*$ and $i\in \{1,2,\dotsc, m\}$ define 
$$N_i(\delta)=\big|\{(a_i,a_j)|a_i\in A_i,\ a_j\in A_j\ (j\neq i), a_i-a_j=\delta\}\big|.$$  
\end{definition}
An AMD code can be described as a game between an {\em encoder} and an {\em adversary}, who is trying to cheat the encoder.
\begin{definition}
A weak $(n,m)$-AMD code is a collection of disjoint subsets $A_1,A_2,\dotsc, A_m$ of an abelian group $G$ with order $n$. Let $k_i$ denote the size of $A_i$, and let $\sum_{i=1}^m k_i=T$.  
\begin{itemize}
\item The encoder picks a source $i$ (number from $1$ to $m$) uniformly at random, and then independently picks an element $g$ uniformly from the set $A_i$.
\item The adversary chooses a value $\delta\in G^*$, and ``succeeds'' if $g+\delta\in A_j$ for $j\neq i$.
\end{itemize}
\end{definition}
Informally speaking, the adversary wins if they can trick the encoder by shifting the group element $g$ into an element $g+\delta$ that is an encoding of a different source than the one that gave rise to the choice of $g$.  A weak $(n,m)$-AMD code is said to be a weak $(n,m,\hat{\epsilon})$-AMD code if $\hat{\epsilon}$ is an upper bound on the success probability of the adversary.  For a given weak AMD code, we observe that the probability that the adversary succeeds when they pick the group element $\delta$ is:
\begin{align}\label{eq:edelta}
e_{\delta}=\frac{1}{m}\left(\frac{1}{k_1}N_1(\delta)+\frac{1}{k_2}N_2(\delta)+\dotsb+\frac{1}{k_m}N_m(\delta)\right).
\end{align}
This expression arises from the fact that a source $i$ is picked with probability $1/m$, and then $N_i(\delta)$ out of the possible $k_i$ encodings of that source will lead to success for an adversary who picks the group element $\delta$.  The overall probability that an adversary succeeds is therefore at most
\begin{align}
\hat{e}=\max_{\delta\in G^*}\frac{1}{m}\left(\frac{1}{k_1}N_1(\delta)+\frac{1}{k_2}N_2(\delta)+\dotsb+\frac{1}{k_m}N_m(\delta)\right),\label{eq:epsilonhat}
\end{align}
and so the AMD code is a weak $(n,m,\hat{\epsilon})$-AMD code for the value of $\hat{\epsilon}$ given in \eqref{eq:epsilonhat}.

In order to obtain lower bounds on $\hat{\epsilon}$ for $(n,m)$-AMD codes, Paterson and Stinson considered the success probability of an attacker who chooses $\delta$ uniformly at random from $G^\ast$\cite{PatSti}.  The success of such an attacker can be determined by computing the average of $\epsilon_{\delta}$ over all choices of $\delta\in G^\ast$:

\begin{align}
\overline{e_{\delta}}&=\frac{1}{n-1} \sum_{\delta\in G^\ast} \frac{1}{m}\left(\frac{1}{k_1}N_1(\delta)+\frac{1}{k_2}N_2(\delta)+\dotsb+\frac{1}{k_m}N_m(\delta)\right),\nonumber\\
&=\frac{1}{m(n-1)}\left(\frac{1}{k_1}\sum_{\delta\in G^\ast}N_1(\delta)+\frac{1}{k_2}\sum_{\delta\in G^\ast}N_2(\delta)+\dotsb+\frac{1}{k_m}\sum_{\delta\in G^\ast}N_m(\delta)\right),\nonumber\\
&=\frac{1}{m(n-1)}\left(\frac{1}{k_1}k_1\sum_{i\neq 1}k_i+\frac{1}{k_2}\sum_{i\neq 2}k_i+\dotsb+\frac{1}{k_m}\sum_{i\neq m}k_i\right),\nonumber\\
&=\frac{1}{m(n-1)}\left(m\sum_{i=1}^m k_i -\sum_{i=1}^m k_i \right),\nonumber\\
&=\frac{(m-1)\sum_{i=1}^n k_i}{m(n-1)}.\label{eq:rbound}
\end{align}

If we set $T=\sum_{i=1}^m k_i$ then the expression in \eqref{eq:rbound} gives the following lower bound for $\hat{\epsilon}$:
\begin{align}\label{eq:actualrbound}
\hat{\epsilon}\geq \frac{(m-1)T}{m(n-1)}.
\end{align}

Paterson and Stinson refer to \eqref{eq:actualrbound} as the {\em random bound,} or {\em R-bound},  and refer to a weak AMD code for which this bound is tight as an {\em R-optimal} weak AMD code.  A weak AMD code is R-optimal precisely when the maximum success probability the adversary has over all possible $\delta\in G^\ast$ is equal to their average success probability.  This gives rise to the following observation:

\begin{theorem}\label{thm:Roptimal}
A weak AMD code is R-optimal if and only if $e_\delta$ is constant for all $\delta\in G^\ast$.
\end{theorem}

In what follows, we will obtain a combinatorial characterization of codes that are optimal in this sense.  Recall the following definition (introduced in \cite{Oga}):
\begin{definition}
An $(n,m,k,\lambda)$-EDF is a set of $m$ disjoint $k$-subsets $A_1,A_2,\dotsc,A_m$ of an abelian group $G$ of order $n$ with the property that 
\begin{align*}
N_1(\delta)+N_2(\delta)+\dotsb+N_m(\delta)=\lambda
\end{align*}
for all $\delta\in G^*$.
\end{definition}

Further definitions were introduced in \cite{PatSti}:
\begin{definition}
\begin{itemize}
\item An {\em $(n,m,k,\lambda)$-SEDF} is an EDF that satisfies the stronger property that
\begin{align*}
N_i(\delta)=\lambda
\end{align*}
for all $i$ from $1,\dotsc,m$ and all $\delta\in G^*$.  In particular, it is an $(n,m,k,m\lambda)$-EDF.
\item An {\em $(n,m;k_1,\dotsc,k_m;\lambda_1,\dotsc,\lambda_m)$-GSEDF} (generalised strong EDF) is a set of $m$ disjoint subsets $A_1,A_2,\dotsc,A_m$ of an abelian group $G$ of order $n$ such that $|A_i|=k_i$ for $1 \leq i \leq m$, and such that
\begin{align*}
N_i(\delta)=\lambda_i
\end{align*}
for all $\delta\in G^*$ and $1\leq i\leq m$.
\end{itemize}
\end{definition}

All of these structures - EDFs, SEDFs and GSEDFs -  have been investigated in the literature because they provide examples of R-optimal AMD codes.  They are R-optimal because the conditions imposed on $N_i(\delta)$ in each definition lead to a constant value of $e_{\delta}$ in Theorem \ref{thm:Roptimal}.  However, we may consider a more general class of combinatorial structure which guarantees R-optimality and allows the potential for new types of code not already covered by the existing, more specialised, definitions.

We begin by making the following new definition:
\begin{definition}
Let $w_1,w_2,\dotsc,w_m$ be {\em weights} that satisfy $0<w_i\leq 1 $, $w_i\in \mathbb{Q}$ for $i\in\{1,\dotsc,m\}$. An $(n,m;k_1,k_2,\dotsc,k_m;w_1,w_2\dotsc,w_m;\ell)$-{\em weighted EDF} (WEDF) is a collection of disjoint subsets $A_1,A_2,\dotsc, A_m$ of an abelian group $G$ with order $n$, where $|A_i|=k_i$ for $i\in\{1,2,\dotsc,m\}$, with the property that 
\begin{align*}
{w_1}N_1(\delta)+{w_2}N_2(\delta)+\dotsb+
{w_m}N_m(\delta)=\ell
\end{align*}
for all $\delta\in G^*$. (Note that $\ell$ is a rational number which need not be an integer.)
\end{definition}

\begin{example}\label{ex:z8}
Consider the subsets $A_1=\{0,1,3\}$, $A_2=\{4,5,7\}$ and $A_3=\{2,6\}$ in $G=\mathbb{Z}_8$.  For $\delta=4$, we have $N_1(4)=N_2(4)=3$ while $N_3(4)=0$.  For any $\delta \in G^* \setminus \{4\}$, $N_1(\delta)=N_2(\delta)=N_3(\delta)=2$.  We observe that $\frac{1}{2} N_1(4)+\frac{1}{2} N_2(4)+ \frac{1}{2} N_3(4)=3.\frac{1}{2}+3.\frac{1}{2}+0.\frac{1}{2}=3$, while for any $\delta \in G^* \setminus \{4\}$ we have
 $\frac{1}{2} N_1(\delta)+\frac{1}{2} N_2(\delta)+ \frac{1}{2} N_3(\delta)=2.\frac{1}{2}+2.\frac{1}{2}+2.\frac{1}{2}=3$.  Hence these subsets form a  $(8,3;3,3,2;\frac{1}{2},\frac{1}{2}, \frac{1}{2};3)$-WEDF.
 \end{example}

\begin{example}
\begin{itemize}
\item  An $(n,m,k,\lambda)$-EDF is an $(n,m; k, \ldots, k; w, \ldots, w; \lambda w)$-WEDF for any choice of weight $w$.
\item An $(n,m,k,\lambda)$-SEDF is an $(n,m;k,\dotsc,k;w_1,w_2,\dotsc,w_m;\ell)$-WEDF for any choice of weights $w_1,w_2,\dotsc,w_m$; here $\ell=\lambda\sum_{i=1}^m w_i$.
\item An $(n,m; k_1, \ldots, k_m; \lambda_1, \ldots, \lambda_m)$-GSEDF is an $(n,m; k_1, \ldots, k_m; w_1, \ldots, w_m; \ell)$-WEDF for any choice of weights $w_1, \ldots, w_m$; here $\ell=\sum_{i=1}^m w_i \lambda_i$.
\end{itemize}
\end{example}

Motivated by a desire to study R-optimal AMD codes, we are particularly interested in the following special case:
\begin{definition}\label{def:RWEDF}
An $(n,m;k_1,\dotsc,k_m;\ell)$-{\em reciprocally weighted EDF} (RWEDF) is an\\ \mbox{$(n,m;k_1,k_2,\dotsc,k_m;w_1,w_2\dotsc,w_m;\ell)$}-WEDF in which the weights $w_i$ are given by $w_i=1/k_i$ for each $i$, so
\begin{align*}
\ell=\frac{1}{k_1}N_1(\delta)+\frac{1}{k_2}N_2(\delta)+\dotsb+
\frac{1}{k_m}N_m(\delta)
\end{align*}
 for each $\delta\in G^*$.
\end{definition}
When viewed as an AMD code, an RWEDF satisfies $e_\delta=\ell/m$ for any $\delta\in G^*$.  It follows that an RWEDF is an R-optimal AMD code. In fact, 

\begin{theorem}
An AMD code is R-optimal {\em precisely} when it is an RWEDF.
\end{theorem}
\begin{proof}
This follows immediately from Theorem~\ref{thm:Roptimal} and Equation~\eqref{eq:actualrbound}.
\end{proof}

We exhibit some known examples of RWEDFs:
\begin{example}\noindent \label{RWEDFex}
\begin{itemize}
\item Two RWEDFs that always exist for any group $G$ are the $(n,1;n;0)$-RWEDF consisting of the whole group, and the $(n,n;1,\dotsc,1;n)$-RWEDF comprising all singletons.  We refer to these as {\em trivial RWEDFs}.
\item For a group $G$, its non-zero elements, taken as singletons, form an $(n,n-1;1, \ldots, 1; n-2)$-RWEDF.
\item An $(n,m,k,\lambda)$-EDF can be viewed as an $(n,m;k,k\dotsc,k;\frac{\lambda}{k})$-RWEDF.
\item An $(n,m,k,\lambda)$-SEDF is an $(n,m;k,k\dotsc,k;\frac{m\lambda}{k})$-RWEDF.
\item An $(n,m; k_1, \ldots, k_m; \lambda_1, \ldots, \lambda_m)$-GSEDF is an $(n,m; k_1, \ldots, k_m; \sum_{i=1}^m \frac{\lambda_i}{k_i})$-RWEDF.
\end{itemize}
\end{example}

\begin{example}\label{ex:z10}\cite{PatSti}
Consider the subsets $A_1=\{0\}$, $A_2=\{5\}$, $A_3=\{1,9\}$, $A_4=\{2,3\}$ in $\mathbb{Z}_{10}$.  We observe that $N_1(5)=N_2(5)=1$ and $N_3(5)=N_4(5)=0$, so $N_1(5)+N_2(5)+\frac{1}{2}N_3(5)+\frac{1}{2}N_4(5)=2$.  For $\delta=2$ we have $N_1(2)=0$, $N_2(2)=1$, $N_3(2)=0$ and $N_4(2)=2$, so $N_1(2)+N_2(2)+\frac{1}{2}N_3(2)+\frac{1}{2}N_4(2)=2$.  Repeating these calculations for the remaining values of $\delta$ will show that these subsets form a $(10,4;1,1,2,2;2)$-RWEDF.  Observe that this is not an EDF, SEDF nor GSEDF.
\end{example}

\begin{remark}
In the literature, AMD codes and difference families have traditionally been defined in the context of an abelian group $G$.  However, all of the definitions stated above (for EDF, WEDF and RWEDF) remain valid when $G$ is an arbitrary finite group $G$, not necessarily abelian.  Although we shall generally focus on the traditional setting where $G$ is abelian, we shall allow the concept of RWEDF to be meaningful for non-abelian $G$, and at certain points in the paper we shall consider existence and constructions of RWEDFs in non-abelian groups.
\end{remark}

\section{Basic results on RWEDFs}
In this section, we summarize basic results that the parameters of any RWEDF must fulfil.
As usual, let $T=\sum_{i=1}^m k_i$.
\begin{theorem}\label{thm:params}
The parameters of an $(n,m;k_1,\dotsc,k_m;\ell)$-RWEDF satisfy
\begin{align}\label{eq:ellsize}(n-1)\ell=(m-1)T.
\end{align} 
\end{theorem}
\begin{proof}
We observe that the number of ways of choosing a pair $(a_i,a_j)$ with $a_i\in A_i$ and $a_j\in A_j$ for some $j\neq i$ is $k_i(T-k_i)$.  Hence, for any $i$, the sum $\sum_{\delta\in G^*} N_i(\delta)$ is equal to $k_i(T-k_i)$.
Now, for each $\delta\in G^*$ we have
\begin{align*}
\ell&=\frac{1}{k_1}N_1(\delta)+\frac{1}{k_2}N_2(\delta)+\dotsb+
\frac{1}{k_m}N_m(\delta)\intertext{so}
(n-1)\ell&=\sum_{\delta\in G^*}\left(\frac{1}{k_1}N_1(\delta)+\frac{1}{k_2}N_2(\delta)+\dotsb+
\frac{1}{k_m}N_m(\delta)\right),\\ 
&=\frac{1}{k_1}\sum_{\delta\in G^*}N_1(\delta)+\frac{1}{k_2}\sum_{\delta\in G^*}N_2(\delta)+\dotsb+\frac{1}{k_m}\sum_{\delta\in G^*}N_m(\delta),\\ 
&=(T-k_1)+(T-k_2)+\dotsb+(T-k_m),\\
&=(m-1)T.
\end{align*}

\end{proof}
From this we derive the following corollary:
\begin{corollary}\label{l<m}
For a nontrivial $(n,m;k_1,\dotsc,k_m;\ell)$-RWEDF we have $\ell<m$, and if $\ell$ is an integer then $\ell\leq m-1$.
\end{corollary}
\begin{proof}
We observe that $T\leq n$, so
\begin{align*}
\ell&=\frac{(m-1)T}{n-1},\\
&\leq \frac{(m-1)n}{n-1},\\
&\leq (m-1)\left(1+\frac{1}{n-1}\right),\\
&\leq (m-1)+\frac{m-1}{n-1}.
\end{align*}
\end{proof}

\begin{lemma}\label{k_i}
For a non-trivial  $(n,m;k_1,\dotsc,k_m;\ell)$-RWEDF,
\begin{enumerate}
\item[(i)] for any $\delta\in G^*$ and any $i$ from $1$ to $m$ we have $N_i(\delta)\leq \min(k_i,T-k_i)$;
\item[(ii)] the number of $\delta\in G^*$ for which $N_i(\delta)\neq 0$ is at least $\max(k_i,T-k_i)$.
\end{enumerate}
\end{lemma}
\begin{proof}
Let $A$ be the $(T-k_i)\times k_i$ array with columns indexed by the elements of $A_i$ and rows indexed by the elements of $\bigcup_{j\neq i} A_j$ where each cell entry is given by the difference between the column label and the row label (i.e. the subtraction table).  Results (i) and (ii) follow immediately from the observation that the entries in each row are distinct, as are the entries in each column. 
\end{proof}

\section{RWEDFs with $m=2$}\label{m=2}
We begin by resolving the situation for RWEDFs with $m=2$; it turns out that these are familiar combinatorial objects.  If $|G|=2$, the situation is trivial; we therefore assume $n>2$.  

By Theorem \ref{thm:params}, an RWEDF with $m=2$ must satisfy $\ell=\frac{T}{n-1}$, where $T=k_1+k_2$. In particular, since $T \leq n$, the only possibility for $\ell \in \mathbb{Z}$ is when $T=n-1$, i.e. when the RWEDF partitions all-but-one of the elements of $G$.  In this case $\ell=1$.

\begin{theorem}\label{m=2RWEDF}
An $(n,m; k_1, \ldots, k_m; \ell)$-RWEDF with $m=2$ is either an EDF or a GSEDF.  \\
Specifically, an $(n,2; k_1,k_2; \frac{k_1+k_2}{n-1})$-RWEDF is an $(n,2,k; \frac{2k^2}{n-1})$-EDF or it is an $(n, 2, k_1, k_2; \frac{k_1 k_2}{n-1})$-GSEDF. 
\end{theorem}
\begin{proof}
From the discussion above, $\ell=\frac{k_1 + k_2}{n-1}$.

If $k_1=k_2=k$, then the RWEDF is an $(n,2,k,k\ell)$-EDF.  In this case, $\ell=\frac{2k}{n-1}$, so $k \ell=\frac{2k^2}{n-1}$.  Since by definition $k \ell$ must be a integer, $n-1$ must divide $2k^2$.

We now assume $k_1\neq k_2$.  We observe that whenever $\delta$ occurs as a difference of the form $a_2-a_1$ with $a_1\in A_1$ and $a_2\in A_2$ then $-\delta$ occurs as the difference $a_1-a_2$.  It follows that $N_2(\delta)=N_1(-\delta)$ for all $\delta\in G^*$.  We have that
\begin{align*}
\ell&=\frac{1}{k_1}N_1(\delta)+\frac{1}{k_2}N_2(\delta),\\
&=\frac{1}{k_1}N_1(\delta)+\frac{1}{k_2}N_1(-\delta).  \intertext{Replacing $\delta$ by $-\delta$ in the above argument gives}
\ell&=\frac{1}{k_1}N_1(-\delta)+\frac{1}{k_2}N_1(\delta), \intertext{so}
\left(\frac{1}{k_1}-\frac{1}{k_2}\right)N_1(\delta)&=\left(\frac{1}{k_1}-\frac{1}{k_2}\right)N_1(-\delta)
\end{align*}
for all $\delta\in G^{*}$.
Since $k_1\neq k_2$ this implies $N_1(\delta)=N_1(-\delta)=N_2(\delta)$.  This implies that
\begin{align*}
\ell&=\left(\frac{1}{k_1}+\frac{1}{k_2}\right)N_1(\delta),\intertext{and hence for any $\delta$}
N_1(\delta)&=\frac{\ell k_1 k_2}{k_1+k_2},\\
&=\frac{k_1k_2}{n-1}
\end{align*}
by Theorem~\ref{thm:params}.  The same is true for $N_2(\delta)$.\\ 
Hence in this case, the RWEDF is an $\left(n,2;k_1,k_2;\frac{k_1k_2}{n-1},\frac{k_1k_2}{n-1}\right)$-GSEDF.

By Example \ref{RWEDFex}, any $(n,2,k; \frac{2k^2}{n-1})$-EDF or $\left(n,2;k_1,k_2;\frac{k_1k_2}{n-1},\frac{k_1k_2}{n-1}\right)$-GSEDF is an RWEDF with $m=2$.
\end{proof}

EDFs have been studied for some time and various constructions are known; recently, GSEDFs have also received attention, for example in  \cite{LuNiuCao} and \cite{WenEtAl}.  In \cite{LuNiuCao}, it is shown that any $(n,2;k_1,k_2;\lambda_1, \lambda_2)$-GSEDF must have $\lambda_1=\lambda_2$ ($=\lambda$, say) where $\lambda | k_1 k_2$, and constructions are given for various $(n,2;k_1, k_2; \lambda, \lambda)$ via a recursive technique.  Many of these constructions satisfy $\lambda=\frac{k_1 k_2}{n-1}$ and so provide infinite families of such RWEDFs.

One natural situation to consider is when the elements of an RWEDF partition $G$ or $G^*$.  These cases have been well-studied for GSEDFs and EDFs; see \cite{PatSti}, \cite{WenEtAl} and \cite{LuNiuCao}.  The following theorem summarizes the results for GSEDFs:

\begin{theorem}\label{GSEDFpartition}
Let $G$ be a finite abelian group and let $\mathcal{A}=\{A_1, \ldots, A_m\}$ ($m \geq 2$) be a collection of disjoint subsets of $G$, with sizes $k_1,k_2,\dotsc,k_m$ respectively.  Then
\begin{itemize}
\item if $\mathcal{A}$ partitions $G$, then $\mathcal{A}$ is an $(n,m;k_1, \ldots, k_m: \lambda_1, \ldots, \lambda_m)$-GSEDF  if and only if each $A_i$ is an $(n, k_i, k_i-\lambda_i)$ difference set in $G$;
\item  if $\mathcal{A}$ partitions $G^*$, then $\mathcal{A}$ is an $(n,m;k_1, \ldots, k_m: \lambda_1, \ldots, \lambda_m)$-GSEDF  if and only if each $A_i$ is an $(n, k_i, k_i-\lambda_i-1, k_i- \lambda_i)$ partial difference set in $G$.
\end{itemize}
\end{theorem}
It is well-known that the complement of an $(n,k,\lambda)$ difference set in a group $G$ is an $(n,n-k,n-2k+\lambda)$ difference set, and it may be shown \cite{WenEtAl} that the complement in $G^*$ of an $(n,k;\lambda, \mu)$ partial difference set is an $(n,n-k-1,n-2k+\mu-2, n-2k+\lambda)$ partial difference set in $G$.

We can characterize the situation in which our $(n,2;k_1,k_2; \ell)$-RWEDF partitions $G$:

\begin{theorem}\label{thm:bigell}
Let $G$ be a group and let $\mathcal{A}=\{ A_1, A_2 \}$ partition $G$.\\   
Then $\mathcal{A}$ is an $(n,2; k, n-k ; \ell)$-RWEDF if and only if $A_1$ is an $(n,k, \lambda)$ difference set and $A_2$ is an $(n,n-k, n-2k+\lambda)$ difference set in $G$.\\
For such an RWEDF, $\ell=\frac{n}{n-1}$; in particular, $l \in \mathbb{Q} \setminus \mathbb{Z}$.
\end{theorem}
\begin{proof}
From Theorem \ref{m=2RWEDF}, an $(n,2; k_1, k_2; \ell)$-RWEDF is either an EDF (when $k_1=k_2$) or a GSEDF (when $k_1 \neq k_2$).  If $k_1=k_2$, i.e. $n$ even and $k=\frac{n}{2}$, then since the number of ordered pairs which give external differences is $n. \frac{n}{2}$, and this is not divisible by $|G^*|=n-1$ since $n$ and $n-1$ are coprime, $\mathcal{A}$ cannot be an EDF.   So $k_1 \neq k_2$ and $\mathcal{A}$ is an GSEDF.  By Theorem \ref{GSEDFpartition} we see that $A_1$ and $A_2$ (which is the complement of $A_1$ in $G$) must be difference sets with the given parameters.  Conversely, it is straightforward to check that if $A_1$ is an $(n, k, \lambda)$ difference set (and so $A_2$ is an $(n, n-k, n-2k+\lambda)$ difference set) then $\mathcal{A}$ is an RWEDF with $\ell=\frac{n(k-\lambda)}{k(n-k)}=\frac{n}{n-1}$.
\end{proof}
We note that the value of $\ell$ attained by the construction of Theorem~\ref{thm:bigell} is the largest possible for any RWEDF in a group of order $n$ when $m=2$.
\begin{example}
Let $G=\mathbb{Z}_7$.  Let $A_1=\{0,1,3\}$ and $A_2=\{2,4,5,6\}$.  Then $\{A_1, A_2 \}$ is a $(7,2; 3,4; \frac{7}{6})$-RWEDF.
\end{example}

Difference sets have been widely studied, and many examples are known.  Since difference sets are defined and known for non-abelian groups, this gives a construction method for non-abelian RWEDFs.

\begin{example}
Let $G$ be the non-abelian group (written multiplicatively) given by $G=\{a,b: a^7=1, b^3=1, bab^{-1}=a^2\}$.  Then a $(21,5,1)$ difference set is given by $D=\{1,a,a^3,b,a^2b^2\}$.  Hence taking $A_1=D$ and $A_2=G \setminus D$ yields a $(21,2; 5, 16; \frac{21}{20})$-RWEDF.
\end{example}

As noted previously, the situation when an $(n,2; k_1, k_2; \ell)$-RWEDF partitions all-but-one of the elements of $G$ is the only case in which the parameter $\ell$ can be an integer; in this case, $\ell=1$.  When an external difference family partitions all-but-one of the elements of $G$ (usually the set of non-zero elements, $G^*$), it is called \emph{near-complete}.  Near-complete EDFs and GSEDFs have received attention in the literature, and some constructions for these offer infinite families of $(n,2; k_1, k_2; 1)$-RWEDFs.  We exhibit a classic example of a cyclotomic construction (see for example \cite{DavHucMul}); cyclotomy is a fruitful construction method in this area.

\begin{example}
For a prime power $q$ congruent to $1$ modulo $4$, let $G$ be the additive group of $GF(q)$.   Take $A_1$ to be the set of squares in $GF(q)^*$ and $A_2$ to be the set of non-squares in $GF(q)^*$; then $\mathcal{A}=\{ A_1, A_2 \}$ is a $(q,2,\frac{q-1}{2}, \frac{q-1}{2})$-EDF and hence a $(q,2;\frac{q-1}{2}, \frac{q-1}{2}; 1)$-RWEDF.
\end{example}


The following constructions (see \cite{PatSti} and \cite{LuNiuCao}) yield $(n,2; k_1, k_2;l)$-RWEDFs that do not partition the whole group.  For $k_1,k_2>2$, these give non-integer values of $l$.
\begin{construction}\label{con:m2optimal}
Consider the sets $\{0,1,2,\dotsc,k-1\}, \{k,2k,\dotsc,k^2\}$.
\begin{itemize}
\item Over $\mathbb{Z}_{k^2+1}$ this forms an SEDF with $\lambda_{\rm SEDF}=1$, and hence a $(k^2+1,2;k,k; \frac{2}{k})$-RWEDF.
\item Over $\mathbb{Z}_{2k^2+1}$ this forms an EDF with $\lambda_{\rm EDF}=1$, and hence a $(2k^2+1,2;k,k;\frac{1}{k})$-RWEDF.
\end{itemize}
\end{construction}

\begin{construction}
Consider the sets $\{0,1,2,\dotsc,k_1-1\},\{k_1,2k_1,\dotsc,k_1k_2\}\subset{\mathbb{Z}_{k_1k_2+1}}$.  This is a GSEDF, which forms a $(k_1 k_2+1, 2; k_1, k_2; \frac{1}{k_1}+\frac{1}{k_2})$-RWEDF.  Observe that 
we can take any values of $k_1$ and $k_2$.
\end{construction}

When using an RWEDF as a weak AMD code, the adversary's success probability is determined by the value of $\ell$.  Hence, in order to find codes where this probability is as small as possible, it is desirable to understand how small $\ell$ can be.  When $m=2$ we have $\ell=(m-1)T/(n-1)=(k_1+k_2)/(n-1)$.  The following theorem establishes the minimum possible value of $\ell$ for RWEDFs with $m=2$.
\begin{theorem}\label{thm:m2bestell}
If there exists an $(n,2;k_1,k_2;\ell)$-RWEDF then $\ell\geq \sqrt{2/(n-1)}$.
\end{theorem}
\begin{proof}
Suppose ${\cal A}=\{A_1,A_2\}$ is an $(n,2;k_1,k_2;\ell)$-RWEDF in a group $G$.  As each element of $G^*$ occurs at least once as a difference of the form $a_i-a_j$ with $a_i\in A_i$, $a_j\in A_j$ and $i\neq j$ we have $2k_1k_2\geq n-1$.  This implies that $k_2\geq (n-1)/(2k_1)$, so 
\begin{align*}
\ell&=\frac{k_1+k_2}{n-1}\\ &\geq \frac{k_1+\frac{n-1}{2k_1}}{n-1}.
\end{align*}
For a fixed value of $n-1$ we can thus minimise $\ell$ by minimising $k_1+\frac{n-1}{2k_1}$.  Treating this as a continuous function of $k_1$, we observe that it has a unique minimum of $\sqrt{2/(n-1)}$, which occurs when $k_1=k_2=\sqrt{(n-1)/2}$.
\end{proof}
The $(2k^2+1,2;k,k;\frac{1}{k})$-RWEDFs of Construction~\ref{con:m2optimal} achieve this minimum value of $\ell$, and hence the bound of Theorem~\ref{thm:m2bestell} is tight.  When used as weak AMD codes with two sources, these RWEDFS are weak $(2k^2+1,2,1/(2k))$-AMD codes in $\mathbb{Z}_{2k^2+1}$, and they exist for any positive integer $k$.  The adversary's success probability can thus be made arbitrarily low at the cost of a quadratic increase in the group size used, and this is best possible.

\section{RWEDF with integer $\ell$}

Although the parameter $\ell$ of an RWEDF may take any rational value, it is natural to begin by considering the case in which $\ell \in \mathbb{Z}$. 

We have seen that it is possible to obtain RWEDFs with $\ell=1$ when $m=2$.  We now give a result which shows that it is possible to obtain RWEDFs with integer $\ell \geq 1$.


\begin{proposition}
Let $G$ be a finite group.  For $1 \leq i \leq m$, let $A_i=\{a_i\}$ where $a_i \in G$.  Then
$A_1, \ldots A_m$ form an $(n,m;1,\ldots,1; \lambda)$-RWEDF if and only if $\{a_1, \ldots, a_m\}$ is an $(n,m,\lambda)$ difference set in $G$.   
\end{proposition}


As noted in the previous section, numerous examples of difference sets are known, in both abelian and non-abelian groups.

The difference set construction rather trivially achieves integer $\ell$ in the equation of Definition \ref{def:RWEDF}, since all the $k_i$'s equal $1$.  A more general condition that would give rise to integer $\ell$ would be the requirement that,  for each $1 \leq i \leq m$, $N_i$ is a multiple of $k_i$.  For a non-trivial RWEDF, we must have $N_i(\delta) \leq k_i$ for all $\delta \in G^*$ by Lemma \ref{k_i}; our requirement would therefore mean that $N_i(\delta) \in \{0,k_i\}$ for all $\delta \in G^*$.  

This motivates the following definition.

\begin{definition}\label{bimodal}
Let $G$ be a finite group and let $\cal A$ be a collection $A_1,A_2,\dotsc,A_m$ of disjoint subsets of $G$ with sizes $k_1,k_2,\dotsc,k_m$ respectively. We shall say that $\cal A$ has the {\em bimodal property} if for all $\delta\in G^*$ we have $N_j(\delta)\in\{0,k_j\}$ for $j=1,2,\dotsc,m$. 
\end{definition}

\begin{remark}
The discussion preceding Definition \ref{bimodal} shows that any bimodal $(n,m; k_1, \ldots, k_m; \ell)$-RWEDF has $\ell \in \mathbb{Z}$.  We note that the converse does not hold;  Example \ref{ex:z10} illustrates an RWEDF with integer $\ell$ that is not bimodal.
\end{remark}

There are some potential parameter choices for an RWEDF that naturally give rise to this bimodal property:

\begin{theorem}\label{thm:allornothing}
An $(n,m;k_1,\dotsc,k_m;\ell)$-RWEDF with $\ell \in \mathbb{Z}$ and $\{k_1, \ldots, k_m\}$ pairwise coprime is bimodal.
\end{theorem}
\begin{proof}
Let $\delta \in G^*$.  By definition, 
\[ \frac{1}{k_1}N_1(\delta)+\frac{1}{k_2}N_2(\delta)+\dotsb+\frac{1}{k_m}N_m(\delta)= \ell. \]
Multiply through by the product $k_1 \cdots k_m$ to get
\[ k_2 \cdots k_m N_1(\delta)+ \cdots + k_1 \cdots k_{m-1} N_m(\delta)=\ell k_1 \cdots k_m, \]
whence
\[ k_2 \cdots k_m N_1(\delta)= k_1 (\ell k_2 \cdots k_m - \cdots - k_2 \cdots k_{m-1} N_m(\delta)). \]
Since $k_1$ divides the right-hand side of this equation, it must divide the left-hand side.  Since $k_1$ is coprime to $k_2, \ldots, k_m$, we must have $k_1 \mid N_1(\delta)$.  If $N_1(\delta)=0$, we are done.  Otherwise, $N_1(\delta)$ is a positive multiple of $k_1$.  But by Lemma \ref{k_i}, $N_1(\delta) \leq k_1$, so in fact $N_1(\delta)=k_1$.

The same argument holds for the other values of $i$.
\end{proof}

Taking the elements of a difference set as singleton sets provides one example of a bimodal RWEDF.  We now exhibit a bimodal RWEDF that satisfies the conditions of Theorem \ref{thm:allornothing}.

\begin{example}\label{ExZ12}
Take $G=\mathbb{Z}_{12}$, $A_1=\{3,6,9\}$, $A_2=\{4,8\}$, $A_3=\{1\}$, $A_4=\{2\}$, $A_5=\{5\}$, $A_6=\{7\}$, $A_7=\{10\}$ and $A_8=\{11\}$.  This is a $(12,8;3,2,1,1,1,1,1,1;7)$-RWEDF that is bimodal. 
\end{example}

We shall investigate how the bimodality property leads to infinite families of new RWEDFs.  We will frequently consider the set-up where we have a collection $\mathcal{A}$ of disjoint subsets $A_1, \ldots, A_m$ of $G$; for each $1 \leq i \leq m$, we will denote by $B_i$ the union $\cup_{j \neq i} A_j$.

Let $I(A_i)$ be the set of internal differences of $A_i$, namely those elements of the form $g_1-g_2$ with $g_1,g_2\in A_i$ and $g_1\neq g_2$.  We will be interested in studying the group that these elements generate.

\begin{definition}
Let $A_i$ be a subset of an abelian group $G$.  We define the {\em internal difference group} of $A_i$ to be the subgroup $H_i\leq G$ that is generated by the elements of $I(A_i)$, namely $H_i=\langle I(A_i)\rangle$.
\end{definition}

\begin{remark}
The group $H_i$ has the property that $A_i$ is contained in a single coset of $H_i$.  Furthermore, $H_i$ is the smallest subgroup $H$ of $G$ with the property that $A_i$ is contained in a single coset of $H$.  To see this, note that by definition, every element of $I(A_i)$ is an element of the group $H_i$.  This implies that for any $u,v\in A_i$ then $u-v\in H_i$ and hence $u$ and $v$ belong to the same coset of $H_i$.  If $H$ is any subgroup of $G$ with $A_i \subseteq x+H$ for some $x\in G$ then every element of $I(A_i)$ lies in $H$, and hence $H_i\leq H$.
\end{remark}

The following theorem characterises the relationship between cosets and bimodality.

\begin{theorem}\label{union_of_cosets}
Let $G$ be a finite abelian group and let $\cal A$ be a collection $A_1,A_2,\dotsc,A_m$ of disjoint subsets of $G$ with sizes $k_1,k_2,\dotsc,k_m$ respectively. Then $\cal A$ is bimodal if and only if for each $j$ with $k_j>1$ the set $B_j$ is a union of cosets of the subgroup $H_j$.
\end{theorem}
\begin{proof}
Let $G$ be a finite additive abelian group and let $\cal A$ be a collection $A_1,A_2,\dotsc,A_m$ of disjoint subsets of $G$ with sizes $k_1,k_2,\dotsc,k_m$ respectively.
Suppose $k_j>1$ and consider $a_j\in A_j$. The differences $a_j-b_j$ with $b_j\in B_j$ are all distinct, which implies that if $N_j(\delta)=k_j$ then for each of the $k_j$ elements $a\in A_j$ there exists $b\in B_j$ with $a-b=\delta$. 

Suppose $\cal A$ is bimodal.  Suppose $k_j>1$ and let $\theta\in I(A_j)$.  Then $\theta=a^\prime-a$ for some $a,a^\prime\in A_j$.  Let $b\in B_j$.  If $a-b=\delta$, then from above there must exist $b^\prime \in B_j$ with $a^\prime-b^\prime=\delta$.  From this, we may deduce that $b+\theta=b+a^\prime-a=-\delta+a+a^\prime-a=b^\prime \in B_j$, and hence we deduce that $B_j+ \theta \subseteq B_j$. Furthermore, for any $\theta \in H_j$ we have $B_j + \theta \subseteq B_j$.  This implies that for any $b\in B_j$, the coset $b+H_j \subseteq B_j$, hence $B_j$ is a union of cosets of $H_j$.

Conversely, suppose that for each $j$ with $k_j>1$ we have that $B_j$ is a union of $r$ cosets of $H_j$, so that $B_j=\cup_{i=1}^r b_i+H_j$ for some distinct $b_i \in B_j$. Then for $a \in A_j$ the differences $a-b$ for $b\in B_j$ are precisely the elements of $\cup_{i=1}^r (a-b_i)+H_j$. As this is the case for any $a \in A_j$, we deduce that $N_j(\delta)=k_j$ if $\delta\in\cup_{i=1}^r (a-b_i)+H_j$ and $0$ otherwise, and hence $\cal A$ is bimodal.
\end{proof}

\begin{corollary}\label{N_i=0}
Suppose $\mathcal{A}$ is bimodal.  Then for $x \in H_i$, $N_i(x)=0$.
\end{corollary}
\begin{proof}
All differences out of $A_i$ have the form $a-b$ where $a \in A_i$ and $b \in B_i$;  by Theorem \ref{union_of_cosets}, $B_i$ is a union of cosets of $H_i$ and is disjoint from the coset of $H_i$ containing $A_i$.  The elements arising as differences therefore  lie within a union of cosets of $H_i$ which does not include $H_i$ itself. 
\end{proof}

We are now able to show that, in certain circumstances, the difference set construction is the only bimodal construction possible - for example, when $\ell=1$:
\begin{theorem}
Let $m\geq 3$. Let $G$ be a finite abelian group of order $n$ and let $\cal A$ be a collection $A_1,A_2,\dotsc,A_m$ of disjoint subsets of $G$ with sizes $k_1,k_2,\dotsc,k_m$ respectively.  Suppose $\cal A$ has the bimodal property.  Then $\mathcal{A}$ is an $(n,m;k_1,\dotsc,k_m;1)$-RWEDF if and only if $\mathcal{A}$ comprises singleton sets whose elements form an $(n,m,1)$ difference set. 
\end{theorem}
\begin{proof}
The reverse direction is immediate.  For the forward direction, suppose it is not the case that $k_1=k_2=\dotsb=k_m=1$.  Then without loss of generality we can suppose that $k_1\geq 2$.  Let $u,v\in A_1$ with $u\neq v$, and denote $u-v$ by $\varepsilon\in I(A_1)$.  The condition $\ell=1$ implies that there is a unique $j$ with $N_j(\varepsilon)=1$.  By Corollary~\ref{N_i=0} we know that $j\neq 1$.  Let $u^\prime\in A_j$.  Then there exists $v^\prime \in B_j$ with $u^\prime-v^\prime=\varepsilon$.  Furthermore, as $\varepsilon\in H_1$ we know that $-\varepsilon\in H_1$, which implies $N_1(-\varepsilon)=0$.  Hence $v^\prime \in A_k$ for some $k\neq 1,j$.

Let $v-u^\prime=\gamma$.  Then there exists $w\in B_1$ with $u-w=\gamma$.  Observe that $w-u^\prime=(u-\gamma)-(v-\gamma)=u-v=\varepsilon$. Since $j$ is the unique value for which $N_j(\varepsilon)\neq 0$, it must be the case that $w\in A_j$.  Note that as $\varepsilon\neq 0$ we have $w\neq u^\prime$.  But this implies $\varepsilon \in I(A_j)$, which contradicts the fact that $N_j(\varepsilon)\neq 0$, by Corollary~\ref{N_i=0}.
\end{proof}

The next result will prove a useful tool in using bimodality to construct new families of RWEDFs.   

\begin{proposition}\label{key_prop}
Let $G$ be a finite abelian group and let $\cal A$ be a collection $A_1,A_2,\dotsc,A_m$ of disjoint subsets of $G$ satisfying the bimodal property.  Then the following conditions are equivalent:
\begin{itemize}
\item $\mathcal{A}$ is an RWEDF;
\item there exists a constant $\lambda$ such that, for all $\delta \in G^*$, $|\{i: N_i(\delta) \neq 0\}|=\lambda$;
\item there exists a constant $\mu$ such that, for all $\delta \in G^*$, $|\{i: N_i(\delta)=0\}|=\mu$.
\end{itemize}
\end{proposition}
\begin{proof}
For $\delta \in G^{*}$, the term $\frac{1}{|A_i|}N_i(\delta)$ equates to $1$ if $N_i(\delta)=k_i$ and $0$ if $N_i(\delta)=0$.  So 
$$ \frac{1}{|A_1|}N_1(\delta)+\frac{1}{|A_2|}N_2(\delta)+\dotsb+\frac{1}{|A_m|}N_m(\delta)$$
counts the number $\mu_{\delta}$ of $i \in \{1, \ldots, m\}$ such that $N_i(\delta) \neq 0$. By definition, $\mathcal{A}$ is an RWEDF if and only if $\mu_{\delta}$ is constant for all $\delta \in G^*$.  Equivalently, since the number of $i \in \{1, \ldots, m\}$ such that $N_i(\delta) = 0$ is given by $m-d_{\delta}$, we see that $\mathcal{A}$ is an RWEDF if and only if this quantity is constant for all $\delta \in G^*$.
\end{proof}
This means that, given a collection of sets known to be bimodal, checking whether it is an RWEDF is equivalent to checking that every non-zero group element arises as a difference (equivalently, does not arise as a difference) out of the same number of $A_i$'s .

\begin{remark}
Observe that, as a consequence of Corollary \ref{l<m}, $\lambda \leq m-1$ and $\mu \geq 1$ in Proposition \ref{key_prop}.
 \end{remark}

In the bimodal RWEDF of Example \ref{ExZ12}, the sets $A_1, \ldots, A_m$ partition $G^*$.  Motivated by this, we consider the general situation in which the sets $A_1, \ldots, A_m$ partition $G^*$.

\begin{proposition}\label{G*:bimodal}
Let $G$ be a finite abelian group and let $\cal A$ be a collection $A_1,A_2,\dotsc,A_m$ of disjoint subsets of $G$ with sizes $k_1,k_2,\dotsc,k_m$ respectively, which partition $G^*$.  Then $\mathcal{A}$ is bimodal if and only if each $A_i$ with $k_i>1$ satisfies $A_i=H_i ^{*}$.
\end{proposition}
\begin{proof}
$(\Rightarrow)$: First, suppose $\mathcal{A}$ is bimodal.  We first show that, for each $i$ with $k_i>1$, one of the following holds: either $A_i=x+ H_i$ for some $x \not\in H_i$ or $A_i=H_i ^{*}$.  We then rule out the former case.

Let $A_i$ ($k_i>1$) be contained in the coset $x+H_i$ of $H_i$.  Since $B_i$ is disjoint from $A_i$ by definition and is a union of cosets of $H_i$ by Theorem \ref{union_of_cosets}, $B_i$ cannot include any of  the coset $x+H_i$.  Since the elements of $\mathcal{A}$ partition $G^*$, the non-zero elements $N$ of $x+H_i$ must be included in the union of all the $A_i$, i.e. must lie in $A_i$. Since by definition $A_i \subseteq N$, we have $N=A_i$.  If $x+ H_i \neq H_i$, the set $N$ is the whole of  $x+H_i$, while if $x+H_i=H_i$ then $N$ is  $H_i ^*$.  

Now, suppose $A_i=x+H_i$, for some $x \not\in H_i$; so $k_i=|H_i|=h$ (say) where $h \geq 2$.  Then $n=hb$ for some positive integer $b$.  Since $\mathcal{A}$ partitions $G^*$, $B_i = G^* \setminus A_i$, and so $|B_i|=(n-1)-h=h(b-1)-1$.  Since $B_i$ is a union of cosets of $H_i$, $h$ divides $h(b-1)-1$.  However, this is possible only if $h=1$.

$(\Leftarrow)$ Suppose that for $A_i$  ($k_i>1$), we have $A_i=H_i^*$.  Then for such an $A_i$, since $\mathcal{A}$ partitions $G^*$, we must have $B_i=G^* \setminus H_i^*$, and so $B_i$ is a union of cosets of $H_i$.  Theorem \ref{union_of_cosets} now guarantees that $\mathcal{A}$ is bimodal.
\end{proof}

 In fact, we can prove that any collection of sets which partition $G^*$, with the property that all non-singleton sets are subgroups with the zero element removed, will yield an RWEDF; here $G$ may be any finite group, abelian or otherwise.  

\begin{theorem}\label{suff_cond_bimodal}
Let $G$ be a finite group of order $n$ and let $\cal A$ be a collection $A_1,A_2,\dotsc,A_m$ of disjoint subsets of $G$ with sizes $k_1,k_2,\dotsc,k_m$ respectively.  Suppose the sets of $\cal A$ satisfy the following:
\begin{itemize}
\item the $A_i$ partition $G^*$
\item every non-singleton $A_i$ has the form $A_i=S_i^*$ for some subgroup $S_i$ of $G$.
\end{itemize}
Then $\mathcal{A}$ is a bimodal $(n,m; k_1,k_2 \ldots, k_m; m-1)$-RWEDF.
\end{theorem}
\begin{proof}
We first prove that $\mathcal{A}$ is bimodal.  For any subgroup $H$ of a finite additive group $G$, the multiset of differences $H-G=\{h-g: h \in H, g \in G\}$ yields each element of $G$ a total of $|H|$ times.  The multiset of differences $H-(G \setminus H)$ yields each element of $G \setminus H$ a total of $|H|$ times (and each element of $H$ zero times), and so the multiset of differences $H^*-(G \setminus H)$  yields each element of $G \setminus H$ a total of $|H|-1$ times (and each element of $H$ zero times).  

Hence in our setting, for each non-singleton $A_i$, the set of differences out of $A_i(=S_i^*)$ comprises each element of $G \setminus S_i$ a total of $|S_i|-1=k_i$ times and each element of $S_i$ zero times: for $\delta \in G^*$, $N_i(\delta)=k_i$ for $\delta \not\in A_i$ and $N_i(\delta)=0$ for $\delta \in A_i$. 

For any singleton $\{g\}$ with $g \neq 0$, $g-(G \setminus \{g\})$ comprises each element of $G\setminus \{0,g\}$ once each (and $0$ and $g$ not at all). So again for $\delta \in G^*$ we have $N_i(\delta)=k_i$ for $\delta \not\in A_i$ and $N_i(\delta)=0$ for $\delta \in A_i$. 

We now show that $\mathcal{A}$ is an RWEDF.  Since the $A_i$ partition $G^*$, each $\delta \in G^*$ is in a unique $A_j$: then for $1 \leq i \leq n$,  $N_i(\delta)=|A_i|$ for $i \neq j$ and $0$ for $i=j$.  So for each $\delta \in G^*$, its weighted sum receives a contribution of $1$ ($= \frac{1}{|A_i|} |A_i|$) when $i \neq j$, and $0$ when $i=j$, i.e. a total of $m-1$.
\end{proof}

\begin{remark}
For an abelian group $G$, the construction of Theorem \ref{suff_cond_bimodal} gives precisely the situation described in Proposition \ref{G*:bimodal}, since if $A_i$ is a subset of an abelian group $G$, such that $|A_i| \geq 2$, and 
$A_i=H^*$ where $H$ is a subgroup of $G$, then $H_i=H$.  The cardinality requirement is important: if $A_i=H^*$ with $|H|=2$, say $H=\{0,h\}$, then $A_i=\{h\}$  and $H_i=\{0\}$.  But if $|H|$ has size $3$, say $\{0,g,h\}$, then the claim holds, as $A_i=\{g,h\}$ and $H_i$ must contain each of $\{0,g,h\}$ by group properties; a similar argument holds when $|H| \geq 3$.
\end{remark}

\begin{example}\label{ExRWEDF_Z_3}
Let $G=\mathbb{Z}_3 \times \mathbb{Z}_3$.  Let $A_1=\{(1,1), (2,2)\}$, $A_2=\{(0,1), (0,2)\}$, $A_3=\{(1,2), (2,1)\}$ and $A_4=\{(1,0), (2,0)\}$.  Observe that for each $A_i$, the subgroup $H_i$ is precisely $A_i \cup \{0\}$.  The union of the two non-trivial cosets of $H_i$ equals the union of the other $3$ sets $A_j$  with $j \neq i$, where each $A_j$ contains precisely one element of each coset.  Then $\cal A$ $=\{ A_1, A_2, A_3, A_4 \}$ is bimodal.  For $\delta \in G^*$, $N_i(\delta)=2$ for $\delta \not\in A_i$ and $N_i(\delta)=0$ for $\delta \in A_i$ (for each $1 \leq i \leq 4$).  This implies that the collection satisfies the conditions of Proposition \ref{key_prop} with $\lambda=3$ and $\mu=1$ and so $\cal A$ forms a $(9,4;2,2,2,2;3)$-RWEDF (indeed, a $(9,4,2,6)$-EDF).  

$$
\begin{array}{c||cc|cc|cc|cc}
-&11&22&01&02 & 12 & 21 & 10 &20  \\ \hline\hline
11&\emph{00} & \emph{22} & 10 & 12 & 02 & 20 & 01 & 21 \\
22&\emph{11} & \emph{00} & 21 & 20 & 10 & 01 & 12 & 02  \\ \hline
01& 20 & 12 & \emph{00} & \emph{02} & 22 &10 & 21 & 11  \\
02& 21 & 10 & \emph{01} & \emph{00} & 20 & 11 & 22 & 12 \\ \hline
12& 01 & 20 & 11 & 10 & \emph{00} &  \emph{21} & 02 & 22 \\
21& 10 & 02 & 20 & 22 & \emph{12} & \emph{00} & 11 & 01  \\ \hline
10& 02& 21 & 12 & 11 & 01 & 22 & \emph{00} & \emph{20} \\
20& 12 & 01 & 22 & 21 & 11 & 02 & \emph{10} & \emph{00}
\end{array}
$$
\end{example}

Various general constructions may be obtained using different groups and subgroups.  The RWEDF given in Example \ref{ExZ12} is a special case of the following construction:

\begin{construction}
Let $G=(\mathbb{Z}_n,+)$ where $n=p^ {\alpha} q^{\beta}$ for distinct primes $p,q$.  The subgroups isomorphic to $(\mathbb{Z}_{p^{\alpha}},+)$ and $(\mathbb{Z}_{q^{\beta}},+)$, each with the zero element removed, can be taken as $A_1$ and $A_2$, while the remaining non-zero elements may be taken as singleton sets.
\end{construction}

The main challenge in constructing such RWEDFs with interesting parameters is to identify groups with sizeable collections of subgroups that are almost-disjoint in the necessary way.  We introduce a group-theoretic concept that will help us in this.

\begin{definition}
If a finite group $G$ has subgroups $S_1,S_2,\dotsc, S_m$ with the property that $S_1^*,S_2^*,\dotsc , S_m^*$ partition $G^*$, then we will call the collection of subgroups $S_1,S_2,\dotsc, S_m$ a \emph{$*$-partition} of $G$.  A $*$-partition is called \emph{trivial} if $m=1$.
\end{definition}

The topic of $*$-partitions of groups has been studied extensively; see \cite{Zap} for a comprehensive survey.  In the literature, $*$-partitions of groups are referred to simply as \emph{partitions of groups}, but in this paper we will use the name $*$-partition to avoid confusion with partitions of the whole group $G$ by subsets of $G$. 

Any $*$-partition of a group forms a bimodal RWEDF.
This is a special (stronger) case of Theorem \ref{suff_cond_bimodal}; here any singleton elements also satisfy the property that $A_i=S_i^*$ for some subgroup $S_i$.
The question of which finite abelian groups possess a non-trivial $*$-partition was established by Miller in \cite{Mil}.

 \begin{theorem}
The only finite abelian groups $G$ admitting a nontrivial $*$-partition are elementary abelian $p$-groups of order $p^e$, for $p$ prime and $e\geq 2$.  
\end{theorem}

The elementary abelian $p$-groups can be viewed as the additive groups of vector spaces over finite fields, and a $*$-partition of such a group can be viewed as a partition of the vectors into subspaces that intersect only in $\mathbf{0}$.  These are known as {\em vector space partitions}, and have been extensively studied.  (See \cite{heden} for a survey on vector space partitions.)  Every elementary abelian $p$-group of order at least $p^2$, for $p$ prime has at least one non-trivial $*$-partition, as the following well-known construction demonstrates:
\begin{construction}\label{con:desarguesian}
Let $p$ be a prime, let $e\geq 2$ and let $e=ab$ for positive integers $a$ and $b$.  The group $\mathbb{Z}_p^e$ can be viewed as the additive group of the $b$-dimensional vector space over $\gf{p^a}$.  The set of all $1$-dimensional subspaces of this $b$-dimensional vector space forms a vector space partition, which corresponds to a $*$-partition of $\mathbb{Z}_p^e$ into subspaces of order $p^a$.

This construction partitions the $p^n-1$ elements of $(\mathbb{Z}_p^e)^\ast$ into $p^{a(b-1)}+p^{a(b-2)}+\dotsb+p^a+1$ sets of size $p^a-1$.  Explicitly, these are precisely the sets of the form 
\begin{align*}
\{\lambda(x_1,x_2,\dotsc,x_{j},1,0,\dotsc,0)|\lambda \in \gf{p^a}^*\}\subset\gf{p^a}^b
\end{align*}
 for some $j=0,1,2,\dotsc,b-1$ and some $x_1,x_2,\dotsc,x_{j-1}\in \gf{p^a}$.
\end{construction}
The $*$-partitions arising from Construction~\ref{con:desarguesian} have the property that all sets in the partition have the same size; the group is then said to be {\em equally partitioned} \cite{isaacs1973}.  For some choices of $a$, $e$ and $p$ there exist $*$-partitions of $(\mathbb{Z}_p^e)^*$ into sets of size $a-1$ that are not isomorphic to those arising from Construction~\ref{con:desarguesian}; in particular, the case where $e=2a$ has been widely studied due to a connection with the construction of translation planes \cite{andre}.    The bimodal RWEDFs arising from equally partitioned groups are in fact EDFs.  As their sets partition the elements of $G^*$ they are examples of near-complete EDFs.  We note that most of the explicit constructions of near-complete EDFs in the literature have used multiplicative cosets in finite fields and are not bimodal.   It is known, however, that a near-complete EDF is equivalent to a disjoint $(v,k,k-1)$-difference family.  Buratti has shown that many known examples of these, including those of Construction~\ref{con:desarguesian}, can be viewed as special cases of a construction arising from an automorphism group acting semiregularly on the kernel of a Frobenius group \cite{buratti}.

Having seen that partitioning $G^*$ with bimodal collections of sets yields new RWEDFs, we may ask whether the same is true when we partition $G$ in a similar way.  

\begin{proposition}\label{G:bimodal}
Let $G$ be a finite abelian group.  Let  $\mathcal{A}=\{A_1, \ldots, A_m \}$ be a set of disjoint subsets that partition $G$.  Then $\mathcal{A}$ is bimodal if and only if each non-singleton $A_i$ is a coset of $H_i$.
\end{proposition}
\begin{proof}
$(\Rightarrow)$:  Suppose $\mathcal{A}$ is bimodal.
Let $A_i$ ($k_i>1$) be contained in the coset $x+H_i$ of $H_i$.  By Theorem \ref{union_of_cosets}, $B_i$ is a union of cosets of $H_i$, disjoint from $A_i$ by definition.  Since $A_i \cup B_i=G$, $A_i$ must contain the coset $x+H_i$.  But this coset contains $A_i$, so $A_i=x+H_i$.\\
$(\Leftarrow)$: Suppose that, for $A_i$ with $k_i>1$, $A_i$ is a coset of $H_i$.  Since $\mathcal{A}$ partitions $G$, $B_i=G \setminus A_i$ is a union of cosets of $H_i$, and so $\mathcal{A}$ is bimodal by Theorem \ref{union_of_cosets}.
\end{proof}

\begin{theorem}\label{no_cosets}
Let $G$ be a finite abelian group and let $\cal A$ be a bimodal collection $A_1,A_2,\dotsc,A_m$ of disjoint subsets of $G$ that partition $G$, with sizes $k_1,k_2,\dotsc,k_m$ respectively.  If $m>1$ and $k_i>1$ for some $1 \leq i \leq m$, then $\mathcal{A}$ is not an RWEDF.
\end{theorem}
\begin{proof}
By Proposition \ref{G:bimodal}, each non-singleton $A_i$ is a coset of $H_i$.  We can consider each singleton as a coset of $\{0\}$.  Suppose sets $A_1, \ldots, A_m$ are cosets of \emph{distinct} subgroups $S_1, \ldots, S_c$ where $1 \leq c \leq m$.  

The trivial RWEDFs correspond to $m=1$ (when $A_1=G$) and $k_1= \cdots = k_m=1$;  to avoid triviality, we may assume $m>1$, and $k_i>1$ for at least one $i \in \{1, \ldots, m\}$.  Write $A_i=x_i+H_i$ where $H_i \in \{S_1, \ldots S_c\}$.  So $|A_i|=|H_i|$.  Note that several $H_i$ may equal the same $S_j$.  We claim that, for $\delta \in G^*$, $N_i(\delta)=0$ if and only if $\delta \in H_i$. Corollary \ref{N_i=0} guarantees the reverse direction.  The forward direction follows from the fact that $B_i$ is the union of all cosets of $H_i$ except for $H_i$ itself.

Let $U=\cup_{i=1}^c S_i$.  The number of non-zero elements in $U$ is at least $1$ and at most $\sum_{i=1}^m (k_i-1)=n-m$.  Since $m \geq 2$, $1 \leq |U \setminus \{0\}|  \leq n-2$. For the (non-zero) elements $\delta \in U$, there is at least one value of $i \in \{1, \ldots, m\}$ such that $N_i(\delta)=0$.   Correspondingly, the number of elements of $G^*$ which do not lie in $U$ satisfies $1 \leq |G \setminus U| \leq n-2$.  For the elements $\delta \in G \setminus U$,  $N_i(\delta)=k_i>0$ for all $i \in \{1, \ldots, m\}$.   

So, overall, the $n-1$ elements of $G^*$ form two disjoint sets, neither of which is empty: namely those $\delta \in G^*$ for which $|\{i: N_i(\delta) \neq 0\}|<m$, and those $\delta \in G^*$ for which $|\{i: N_i(\delta) \neq 0\}|=m$.  By Proposition \ref{key_prop}, this is not an RWEDF.


\end{proof}

We observe that, although motivated by a necessary condition for abelian groups, the construction of Theorem \ref{suff_cond_bimodal} holds for any finite group $G$.  Hence any collection of subgroups in a non-abelian $G$ which intersect only in the identity, may be used to construct one of these more generally-defined RWEDFs, by taking the subgroups with identity removed, then taking all remaining non-identity elements as singleton sets.

Furthermore, the notion of $*$-partition is defined for any finite group, and an RWEDF can be constructed from a $*$-partition of any such group.  We may ask which finite groups $G$ admit a non-trivial $*$-partition; a characterization is given in \cite{Zap}.

\begin{theorem}[\cite{Zap}]\label{ZapThm}
A finite group $G$ has a non-trivial $*$-partition if and only if it satisfies one of the following conditions:
\begin{itemize}
\item $G$ is a $p$-group with Hughes subgroup $H_p(G) \neq G$ and $|G|>p$;
\item $G$ is a Frobenius group;
\item $G$ is a group of Hughes-Thompson type;
\item $G$ is isomorphic to $PGL(2,p^h)$ with $p$ an odd prime;
\item $G$ is isomorphic to $PSL(2,p^h)$ with $p$ prime;
\item $G$ is isomorphic to a Suzuki group $G(q)$, $q=2^h$, $h>1$.
\end{itemize}
\end{theorem}

It is known that the equally partitioned groups are precisely the $p$-groups of exponent $p$ \cite{isaacs1973}.  Each such group has a $*$-partition into subgroups of order $p$; some of them additionally permit $*$-partions into larger subgroups of equal size, although these have not been fully classfied.  Any equally-sized $*$-partition of a nonabelian $p$-group of exponent $p$ gives rise to a nonabelian EDF. 
\begin{example}
Let $G$ be the set of $3\times 3$ upper triangle matrices with entries from $\gf{3}$ that have $1$s on the main diagonal.  These are closed under multiplication and hence form a (nonabelian) group.  Each element of $G$ has the form
\begin{align*}
\begin{pmatrix}
1&a&b\\0&1&c\\0&0&1\end{pmatrix},&\intertext{and we have that}
\begin{pmatrix}
1&a&b\\0&1&c\\0&0&1\end{pmatrix}^3&= \begin{pmatrix}
1&3a&3b+3ac\\0&1&3c\\0&0&1\end{pmatrix}\equiv \begin{pmatrix}1&0&0\\0&1&0\\0&0&1\end{pmatrix},
\end{align*}
so each non-identity element has order $3$.  There are three choices for each of $a$, $b$ and $c$, and hence $G$ has order 27.  The order 3 subgroups partition its non-identity elements; this will therefore give a near-complete EDF with 13 sets of size 2.
\end{example}
We now give an example of a nonabelian RWEDF that is not an EDF.  Following convention, we use multiplicative rather than additive notation for non-abelian groups.  In particular, $xy^{-1}$ replaces $x-y$ (though for consistency we may still refer to this as the difference when there is no risk of confusion).
\begin{example}
Let $n$ be odd, and let $D_{2n}$ be the dihedral group that is given by the presentation $\{x, y: \mathrm{ord}(x)=n, \mathrm{ord}(y)=2, xy=yx^{-1}\}$.  (This is an example of a Frobenius group.)  A $*$-partition is given by $S_i=\langle yx^{i-1} \rangle$ for $1 \leq i \leq n$ and $S_{n+1}=\langle x \rangle$.  Here $|S_1|= \cdots = |S_n|=2$ and $|S_{n+1}|=n$. 

For $D_{10}=\{x, y: x^5=y^2=1, xy=yx^{-1}\}$, our $*$-partition yields the sets $A_1=\{y\}$, $A_2=\{yx\}$, $A_3=\{yx^2\}$, $A_4=\{yx^3\}$, $A_5=\{yx^4\}$ and $A_6=\{x,x^2,x^3,x^4\}$.  This is a $(10,6; 1,1,1,1,1,4; 5)$- RWEDF.  For each $A_i$ with $1 \leq i \leq 5$, every non-identity element of $D_{10}$ except for the single element of $A_i$ itself, appears once as a difference out of $A_i$, i.e. here $N_i(\delta)=1$ for $\delta \neq yx^{i-1}$ and $N_i(\delta)=0$ for $\delta = yx^{i-1}$. For $A_6$, every element of $y \langle x \rangle$ appears $4$ times as a difference out of $A_6$, i.e. $N_6(\delta)=4$ for $\delta \in y \langle x \rangle$ and $N_6(\delta)=0$ for $\delta \in A_6$. Hence, for $\delta \in D_{10}^*$, if $\delta \in y \langle x \rangle$ then the weighted sum is
$$  0+1\cdot1+1\cdot1+1\cdot1 + 1\cdot1 + \frac{1}{4} 4=5$$
while for $\delta \in \langle x \rangle$ the weighted sum is
$$ 1\cdot1+1\cdot1+1\cdot1+1\cdot1+1\cdot1+0=5.$$
\end{example}

\section{Conclusions and future work}

In this paper, we have introduced the RWEDF as a combinatorial way of viewing AMD codes which are R-optimal.  We have presented various RWEDF constructions, which yield both examples of known structures such as EDFs and SEDFs, and examples of objects not previously seen.   When we focus on the natural situation when the parameter $\ell$ is an integer, the concept of \emph{bimodality} seems to be a useful tool.

Emerging from this work are various very natural questions that remain open.

In Section \ref{m=2}, understanding RWEDFs with $m=2$ is shown to rely on an understanding of GSEDFs with $m=2$.

\begin{question}
Classify the GSEDFs with $m=2$.
\end{question}

The bimodal RWEDFs we have described give new infinite families of RWEDFs with integer $\ell$, but we know that integer $\ell$ does not imply bimodality.

\begin{question}
Find new RWEDFs with $\ell \in \mathbb{Z}$ that are not bimodal.
\end{question}

Although the case when $\ell$ is an integer seems mathematically natural, we can also ask whether it has structural significance for the objects involved.

\begin{question}
Is there a combinatorial characterization of RWEDFs with integer $\ell$?
\end{question}

We have not investigated the situation where $\ell \not\in \mathbb{Z}$ beyond the case of $m=2$.

\begin{question}
Find new RWEDFs with $\ell \in \mathbb{Q} \setminus \mathbb{Z}$ for $m>2$.
\end{question}

It would be especially interesting to find examples that are not EDFs.

\subsection*{Acknowledgements} The first author is supported by a Research Incentive Grant from The Carnegie Trust for the Universities of Scotland (Grant No. 70582).

\bibliography{edfbib}
\end{document}